\documentclass[11pt,a4paper]{amsart}

\usepackage[english]{babel}
\usepackage{amsmath}
\usepackage{amssymb}
\usepackage{amsthm}
\usepackage{latexsym}
\usepackage{mathtools}
\usepackage{thmtools}
\usepackage{amsfonts}
\usepackage{mathrsfs}
\usepackage{textcomp}
\usepackage{graphicx}
\usepackage{setspace}
\usepackage{nicefrac}
\usepackage{indentfirst}
\usepackage{enumerate}
\usepackage{wasysym}
\usepackage[normalem]{ulem}
\usepackage{upgreek}
\usepackage{paralist}
\usepackage{xcolor}
\usepackage{etoolbox}
\usepackage{mathdots}
\usepackage{wrapfig}
\usepackage{floatflt}
\usepackage{tensor} 
\usepackage{parskip}
\usepackage{lscape}
\usepackage{stmaryrd}
\usepackage[enableskew]{youngtab}
\usepackage{ytableau}
\usepackage{pifont}
\usepackage{braket}
\usepackage{derivative}

\usepackage[all,cmtip]{xy}
\usepackage[labelfont=bf, font={small}]{caption}[2005/07/16]

\usepackage{tikz}
\usetikzlibrary{cd}
\usetikzlibrary{arrows}
\usetikzlibrary{matrix}
\usetikzlibrary{positioning}
\usetikzlibrary{calc}
\usetikzlibrary{fit}

\usepackage{pgfplots}
\pgfplotsset{compat=1.17}

\tikzset{%
  highlight/.style={rectangle,rounded corners,fill=red!15,draw,
    fill opacity=0.5,thick,inner sep=0pt}
}

\usepackage[breaklinks,bookmarks=false, hidelinks]{hyperref} 
\hypersetup{
    colorlinks,
    linkcolor={blue!50!black},
    citecolor={green!50!black},
    urlcolor={blue!80!black}
}%

\newcommand{\xto}[1]{\xrightarrow{\phantom{a}{#1}{\phantom{a}}}}

\newcommand{\vvirg}{ , \dots , }

\newcommand{\contract}{\rotatebox[origin=c]{180}{ \reflectbox{$\neg$} }}

\newcommand{\bfT}{\mathbf{T}}

\newcommand{\bbC}{\mathbb{C}}

\newcommand{\bbL}{\mathbb{L}}

\newcommand{\bbP}{\mathbb{P}}

\newcommand{\bbU}{\mathbb{U}}

\newcommand{\bbX}{\mathbb{X}}
\newcommand{\bbY}{\mathbb{Y}}
\newcommand{\bbZ}{\mathbb{Z}}

\newcommand{\frakS}{\mathfrak{S}}

\newcommand{\rmR}{\mathrm{R}}

\renewcommand{\phi}{\varphi}

\renewcommand{\bar}[1]{\overline{#1}}

\newcommand{\rank}{\mathrm{rank}}

\DeclareMathOperator{\codim}{codim}

\DeclareMathOperator{\Sym}{Sym}

\DeclareMathOperator{\Ann}{Ann}

\newcommand{\GL}{\mathrm{GL}}
\newcommand{\SL}{\mathrm{SL}}

\renewcommand{\HF}{h}

\newcommand{\cat}{\mathrm{cat}}

\newtheorem{theorem}{Theorem}[section]

\newtheorem{lemma}[theorem]{Lemma}
\newtheorem{proposition}[theorem]{Proposition}

\theoremstyle{definition}

\usepackage[nameinlink]{cleveref}
\crefname{theorem}{theorem}{theorems}
\Crefname{theorem}{Theorem}{Theorems}
\crefname{lemma}{lemma}{lemmas}
\Crefname{lemma}{Lemma}{Lemmas}
\crefname{proposition}{proposition}{propositions}
\Crefname{proposition}{Proposition}{Propositions}
\crefname{corollary}{corollary}{corollaries}
\Crefname{corollary}{Corollary}{Corollaries}
\crefname{definition}{definition}{definitions}
\Crefname{definition}{Definition}{Definitions}
\crefname{remark}{remark}{remarks}
\Crefname{remark}{Remark}{Remarks}
\crefname{example}{example}{examples}
\Crefname{example}{Example}{Examples}
\crefname{examplex}{example}{examples}
\Crefname{examplex}{Example}{Examples}

\usepackage[top=3cm,bottom=3cm,left=3cm,right=3cm]{geometry}

\renewcommand{\phi}{ \varphi }
\newcommand{\ev}{\mathrm{ev}}
\newcommand{\reg}{\mathrm{reg}}

\usepackage{color}

\title{Waring decompositions of special binomials}
\author{Luca Chiantini}

\author{Fulvio Gesmundo}

\author{Sara Marziali}

\address[L. Chiantini, S. Marziali]{Dipartimento di Ingegneria dell'Informazione e Scienze Matematiche, Universit\`a di Siena, Italy}

\address[F. Gesmundo]{Institut de Mathématiques de Toulouse; UMR5219 -- Université de Toulouse; CNRS -- UPS, F-31062 Toulouse Cedex 9, France}

\email[(Chiantini)]{luca.chiantini@unisi.it}
\email[(Gesmundo)]{fgesmund@math.univ-toulouse.fr}
\email[(Marziali)]{sara.marziali@student.unisi.it}


\begin{document}

\begin{abstract}
We determine the Waring rank of homogeneous polynomials of the form $x^ky^kz^k + \ell^{3k}$ where $\ell$ is a linear form. The result is based on the study of the Hilbert function and the resolution of special configurations of points in $\bbP^2$. As a byproduct of our result, we show that the monomial $x^ky^kz^k$ does not have irredundant decompositions of length $(k+1)^2 +1$. 
\end{abstract}

\maketitle

\section{Introduction}

The Waring rank is a fundamental invariant of homogeneous polynomials, studied for almost 200 years. For a homogeneous polynomial $f$ of degree $d$, the \emph{Waring rank} of $f$, denoted $\rmR(f)$, is the smallest integer $r$ such that $f$ can be written as a linear combination of powers of linear polynomials $\ell_1 \vvirg \ell_r$. Over the field of complex numbers, one has a natural geometric translation: the Waring rank of $f$ is the smallest $r$ such that the point $[f]$ in a corresponding projective space lies on an $r$-secant plane to the $d$-th Veronese variety. The set of linear forms $\{ \ell_1 \vvirg \ell_r\}$, or equivalently the corresponding set of points in projective space, is called a minimal \emph{Waring decomposition} of $f$.

Determining Waring rank for a given polynomial is considered an extremely hard problem. Formal hardness results are proved in \cite{HilLim:TensorsNPhard,Shitov16} and more generally the geometric formulation often translates the problem of determining Waring decompositions into challenging problems in determining ideals of points with subtle properties. The value of the Waring rank for generic polynomials was determined in \cite{AlHir}, and \cite{GalMel} completes the classification of the cases where a generic polynomial has a unique minimal Waring decomposition. Generic elements of Waring rank smaller than the generic rank have a unique decomposition except for few cases \cite{ChiOttVan}, and for small Waring rank the unique decomposition can be recovered efficiently using, for instance, eigenvalue methods, see, e.g. \cite{BCMT10} and \cite[Section 6]{gesmundo2025hilbert}. Understanding the set of decompositions in the cases where the decomposition is not unique is considered hard even for generic polynomials: partial results are known in small degree and number of variables \cite{RanSch:VSPs,MasMelVSPS,RanVoiVSPS}. For special polynomials, the problem is difficult as well: the Waring rank and some properties of Waring decompositions are known only in very special cases, and were determined with specific methods \cite{TeitWoo,Lee16,CCCGW15,ABC22}. An important example is the case of monomials: \cite{CCG12} determined the Waring rank of monomials and \cite{BBT13} provides a characterization of all minimal Waring decompositions, see~\Cref{prop: monomial results}.

In this work, we study the case of polynomials of a special form, that is $f = (x_0 \cdots x_n)^k + \ell^{(n+1)k}$ where $\ell$ is a linear form. We focus on the cases $n=1,2$. The case $n=2$ is already interesting: the study of the Hilbert function of a potential decomposition, together with the results of \cite{CCG12,BBT13}, allows us to obtain the following result.

\begin{theorem}\label{thm: main}
    Let $k \geq 2$. Let $f = x_0^kx_1^kx_2^k + \lambda \ell^{3k}$ for a linear form $\ell = a_0 x_0 + a_1 x_1 + a_2 x_2$ and $\lambda \neq 0$. Then 
    \begin{itemize}
        \item If $a_0a_1a_2 \neq 0$ there is a unique $\lambda_0$ such that $\rmR(f) = (k+1)^2-1$ if $\lambda = \lambda_0$ and $\rmR(f) = (k+1)^2$ if $\lambda \neq \lambda_0$; moreover, for any $\lambda$, $f$ has a unique minimal decomposition.
        \item If $a_0a_1a_2 = 0$ then $\rmR(f) = (k+1)^2 + 1$.
    \end{itemize}
\end{theorem}
\Cref{thm: main} is proved in~\Cref{sec: main proof}. Analogous results for the cases $n=1$ and $(n,k) = (2,1)$, are given in~\Cref{sec: first cases}. In fact, a partial version of the first part of the statement holds in any number of variables, see~\Cref{prop: any vars}.

The polynomial $f$ of~\Cref{thm: main} can be considered as the very first generalization beyond the cases of monomials or sums of disjoint monomials, which were considered in \cite{CCG12}. The Waring rank for binomials in two variables was computed in \cite{BM:Binomials}. 

Polynomials of the form $\ell_1 \cdots \ell_d + \ell_0^d$ define a fundamental model of computation, the product-plus-power model, studied in \cite{Kum19,DGIJL25}. From the point of view of computational complexity, this model is essentially equivalent to border $\Sigma \Lambda \Sigma$ circuit complexity, that is the model of border Waring rank, see \cite[Theorem 1.1]{DGIJL25}. Several debordering results were provided in \cite{DGIJL25}, but a major barrier is that we lack the fundamental understanding of Waring decompositions of (even very simple) binomials.~\Cref{thm: main} is a first step in this direction.

From a different point of view, one is often interested in understanding the existence of possibly non-minimal, irredundant, decompositions of homogeneous polynomials. These are expressions $f = \ell_1^d + \cdots + \ell_r^d$ such that no proper subset of $\{ \ell_1 \vvirg \ell_r\}$ yields a decomposition of $f$, but $r$ is possibly not minimal, in the sense that the inequality $\rmR(f) \leq r$ can be strict. Very little is understood about this setting. The case of quadrics is already interesting: in that case the set of decompositions is parametrized by \emph{Parseval frames}, geometrically encoded by the Stiefel manifolds, see, e.g., \cite{BryGes}. The non-existence of non-minimal, irredundant decompositions for values of $r$ near $\rmR(f)$ was studied in \cite{AC20,AngCh22Minimality,ChiOtt} in special cases. 

Having a good understanding of non-minimal decompositions is useful in several contexts. For instance, for some numerical methods, it is desirable to have a positive-dimensional space of solutions to a given problem, and search in such a space for solutions with special properties; minimal decompositions often arise in a finite number, whereas the relaxation to non-minimal decompositions allows for infinite families of solutions more suitable for numerical search, see, e.g., \cite[Ch. 2]{Wolsey}. This is convenient in cases where one searches for decompositions with other kinds of minimality properties \cite{DerkNuclear} or special geometric features \cite{RHST24}. More generally, the Hilbert function methods used in the present work often exploit the existence of a non-minimal, irredundant decomposition, to construct a minimal one; this was applied effectively in \cite{AngCh22Minimality} using algebraic and geometric properties of the minimal free resolution of ideals of points and liaison in $\bbP^2$. Our proof is based on similar techniques, and provides the following result, of which~\Cref{thm: main} will be a consequence. 
\begin{theorem}\label{thm: main overcomplete}
    Let $g = x^k y^k z^k$. Then $g$ has no irredundant Waring decompositions of length $(k+1)^2 +1$.
\end{theorem}

\section{Preliminaries}
We work over the field of complex numbers. Let $V$ be a vector space of dimension $n+1$ and let $V^*$ be its dual space. The ring of homogeneous polynomials on $V^*$ is identified with the symmetric algebra $\Sym V$: in coordinates, if $x_0 \vvirg x_n$ is a basis of $V$, then $\Sym V = \bbC[x_0 \vvirg x_n]$. The space of homogeneous polynomials of degree $d$ is denoted $S^d V$.

Let $f \in S^d V$. The Waring rank of $f$ is 
\[
\rmR(f) = \min\{ r : f = \ell_1^d + \cdots + \ell_r^d \text{ for some } \ell_1 \vvirg \ell_r \in V\}.
\]
From a geometric point of view, one may consider the Veronese variety $v_d (\bbP V) = \{ [\ell^d] \in \bbP S^d V : \ell \in V\}$, and equivalently define the Waring rank as 
\[
\rmR(f) = \min \{ r : [f] \in \langle v_d(\bbX) \rangle \text{ for a set $\bbX \subseteq \bbP V$ of $r$ points} \}.
\]
A set $\bbX = \{ [\ell_1] \vvirg [\ell_r] \}$ such that $[f] \in \langle v_d(\bbX) \rangle$ is called a \emph{Waring decomposition} of $f$. We say that $\bbX$ is an irredundant decomposition if $\bbX$ is minimal with respect to inclusion, that is there is no $\bbX' \subsetneq \bbX$ such that $[f] \in \langle v_d(\bbX') \rangle$. We say that $\bbX$ is a minimal decomposition, or a rank decomposition, if the Waring rank equals the number of elements of $\bbX$, that is $\rmR(f) = \# \bbX $.

\subsection{Apolarity theory} A fundamental tool in the study of Waring rank and Waring decompositions is \emph{apolarity theory}, which relates the existence of a Waring decomposition of $f \in S^d V$ to the existence of an ideal of points contained in the annihilator of $f$ under an \emph{apolar action}. We provide here an outline of the theory and the main results and we refer to \cite{BCCGO} for more details.

The ring $\Sym V^*$ in the dual variables has two roles: it can be regarded as the homogeneous coordinate ring of the projective space $\bbP V$, and as the ring of differential operators (with constant coefficients) on $\Sym V$. In the first interpretation, one can consider ideals of subvarieties, and in particular finite sets of points, of $\bbP V$ inside $\Sym V^*$. In the second interpretation, every polynomial $f$ has an \emph{apolar ideal} 
\[
\Ann(f) = \{ D \in \Sym V^* : D \contract f  = 0\},
\]
where $D \contract f $ is the result of the differentiation of $f$ by the differential operator $D$. We will use capital letters $X_0 \vvirg X_n$ to denote the basis of $V^*$ dual to the basis $x_0 \vvirg x_n$ of $V$. In particular $X_j (x_i ) = \frac{\partial}{\partial x_j} x_i = \delta _{ij}$. If $F \in \Sym V^*$ then $F \contract f = F( \frac{\partial}{\partial x_0} \vvirg \frac{\partial}{\partial x_n}) f$. 

The homogeneous components of the apolar ideal are characterized as the kernel of a corresponding partial derivative map, classically called \emph{catalecticant} map:
\begin{align*}
\cat_p(f) : S^p V^* &\to S^{d-p} V \\ 
D & \mapsto D \contract f.
\end{align*}

The classical Apolarity Lemma relates the ``two natures'' of the ring $\Sym V^*$, see, e.g., \cite[Lemma 5]{BCCGO} for details.
\begin{lemma}[Apolarity Lemma]\label{lem: apolarity}
Let $f \in S^d V$ and let $\bbX = \{ [\ell_1] \vvirg [\ell_r] \} \subseteq \bbP V$ be a set of points. The following are equivalent:
\begin{itemize}
    \item there exist $\lambda _ 1 \vvirg \lambda_r$ such that $f = \lambda_1 \ell_1^d + \cdots + \lambda_r \ell_r^d$;
    \item the ideal $I(\bbX)$ of $\bbX$ is contained in the apolar ideal $\Ann(f)$.
\end{itemize}
\end{lemma}

We record here a special case of the results of \cite{CCG12, BBT13}, based on apolarity theory, which will be useful in the following.
\begin{proposition} \label{prop: monomial results}
Let $k$ be a positive integer and let $g = (x_0 \cdots x_n)^k \in S^{k(n+1)} V$. Then: 
\begin{enumerate}[(i)]
\item the apolar ideal $\Ann(g)$ of $g$ is 
\[
  \Ann(g) = (X_0^{k+1} \vvirg X_n^{k+1});
\]
\item the Waring rank of $g$ is $\rmR(g) = (k+1)^n$;
\item for every $\alpha_1 \vvirg \alpha_n \neq 0$, the complete intersection ideal 
\[
I(\bbX) = ( X_0^{k+1} - \alpha_1 X_1^{k+1} \vvirg  X_0^{k+1} - \alpha_n X_n^{k+1});
\]
defines a minimal decomposition $\bbX$ of $g$ of length $(k+1)^n$. All minimal decompositions arise in this way.
\end{enumerate}
\end{proposition}
\begin{proof}
Part (i) is immediate. Part (ii) follows from \cite{CCG12}. Part (iii) follows from \cite{BBT13}; see also \cite{CCO17}.
\end{proof}

\subsection{Hilbert functions of sets of points}
The condition that $I(\bbX) \subseteq \Ann(f)$ in~\Cref{lem: apolarity} guarantees, in particular, that, for every integer $t$, one has $\codim \Ann(f)_t \leq \codim I(\bbX)_t$, where the homogeneous components are regarded as subspaces of $S^t V^*$. The Hilbert function of a homogeneous ideal is the object that records this codimension. More precisely, for a homogeneous ideal $I \subseteq \Sym V^*$, the \emph{Hilbert function} of (the algebra associated to) $I$ is 
\begin{align*}
\HF_{ \Sym V^* / I } : \bbZ &\to \bbZ \\ 
t &\mapsto  \codim I_t.
\end{align*}
If $I = I(\bbX)$, write $\HF_\bbX (-)$ for the Hilbert function of $\Sym V^* / I(\bbX)$. If $\bbX = \{ [v_1] \vvirg [v_r]\}$ then one may equivalently define 
\[
\HF_{ \bbX}(t) = \rank ( \ev_\bbX : S^t V^* \to \bbC^r ) 
\]
where $\ev_\bbX$ is the evaluation map at the vectors $v_1 \vvirg v_r$; this rank does not depend on the choice of representatives of the points $[v_1] \vvirg [v_r]$. In particular $\HF_\bbX$ records the number of linearly independent conditions imposed by $\bbX$ on the homogeneous polynomials of degree $d$. If $\HF_{ \bbX} (t) = \#\bbX$, we say that $\bbX$ imposes independent conditions in degree $t$.

The study of Hilbert functions of sets of points in projective spaces is a central theme in algebraic geometry and commutative algebra. We only require few results about this subject, for which we refer to \cite[Sec. 2]{ChiGes} and \cite{chiantini2019hilbert}.

If $\bbX$ is a finite set of points, then $\HF_\bbX$ is strictly increasing until it stabilizes to the value $\#\bbX$; the \emph{regularity} of $\bbX$ is the first degree after the one where $\HF_\bbX$ reaches its stable value, that is $\reg(\bbX) = \min \{ \tau > 0 : \HF_\bbX(\tau) = \HF_\bbX(\tau-1)\}$. In particular, if $\bbX$ is a set of $r$ points then $\reg(\bbX) \leq \#\bbX$.

The \emph{first difference of the Hilbert function} is
\[
D \HF_\bbX (t) = \HF_\bbX (t) - \HF_\bbX (t-1) .
\]
For every set of points $\bbX$, we have $D\HF_\bbX(t) = 0$ if $t \geq \reg(\bbX)$; therefore often $D \HF_\bbX$ is recorded simply as the sequence of its non-zero values and represented pictorially in a diagram of boxes. Since $\sum_{t} Dh_\bbX(t) = \#\bbX$, the total number of boxes in the diagram for $Dh_\bbX$ equals $\#\bbX$. For instance, the blue boxes in~\Cref{fig: example diagram} represent the first difference of the Hilbert function $D h_\bbX = (1,2,1,1)$ of the set of five points 
\[
\bbX = \{(1,0,0), (1,1,0) , (1,2,0), (0,1,0), (0,0,1) \} \subseteq \bbP^2.
\] 
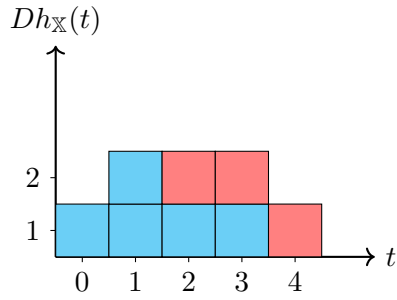
\begin{figure}[!htp]
    \centering
\begin{tikzpicture}[scale=0.7] 
    \draw[->, thick] (0,0) -- (6,0) node[right] {$t$};
    \draw[->, thick] (0,0) -- (0,4) node[above] {$Dh_{\mathbb X}(t)$};

    \foreach \x in {0,...,4} \draw (\x + 0.5,0) -- (\x + 0.5,-0.1) node[below] {$\x$};
    \foreach \y in {1,...,2} \draw (0,\y - 0.5) -- (-0.1,\y -0.5) node[left] {$\y$};
    
    \foreach \x/\y in {0/1, 1/1, 2/1, 3/1, 1/2} 
    {
        \fill[draw=black, fill=cyan!50] (\x, \y-1) rectangle (\x+1, \y);
    }

    \foreach \x/\y in {2/2, 3/2, 4/1}
    {
        \fill[draw=black, fill=red!50] (\x, \y-1) rectangle (\x+1, \y);
    }
\end{tikzpicture}
    \caption{In blue, the diagram of $D h_\bbX = (1,2,1,1)$ for the set $\bbX = \{(1,0,0), (1,1,0) , (1,2,0), (0,1,0), (0,0,1) \}$. In red, the diagram of a linked set $\bbY$ of three collinear points. The union of the two diagrams represents $D h_{\bbX \cup \bbY}$, where $\bbX \cup \bbY$ is a complete intersection $(2,4)$.}
    \label{fig: example diagram}
\end{figure} 

The following result is a special case of \cite[Thm. 3.6]{bigatti1994geometric}; we will use it only in the case $n = 2$, in which case the original result already appears in \cite{Davis}.
\begin{theorem}\label{thm: bgm}
    Let $\bbX \subseteq \bbP V$ be a finite set of points. Let $s \leq t_0$ and suppose $s = D \HF_\bbX (t_0) = D \HF_\bbX (t_0 + 1)$. Then there exists a reduced curve $C \subseteq \bbP V$ of degree $s$ such that
    \begin{itemize}
        \item $D \HF_{\bbX \cap C} (t) = D\HF_{C} ( t)$  for $t \leq t_0 + 1$;
        \item $D \HF_{\bbX \cap C} (t) =D \HF_{\bbX } (t)$  for $t \geq t_0$.
    \end{itemize}
\end{theorem}
If $\bbX$ satisfies $D \HF_\bbX (t_0) =  D \HF_\bbX (t_0+1)$, we say that $D \HF_\bbX$ has a plateau at height $s = D \HF_\bbX (t_0)$. For instance, the set of points $\bbX$ corresponding to the blue boxes of~\Cref{fig: example diagram} has a plateau at height $1$ starting at degree $2$.~\Cref{thm: bgm} guarantees that if $\bbX$ has a plateau, then a subset of the points of $\bbX$ satisfy the strong condition of lying on a curve of small degree; observe that indeed the set of points of~\Cref{fig: example diagram} has four points lying on a line. This fact is often used in the contrapositive: if many points lie on a curve of small degree, one derives a contradiction, and this allows one to conclude that the Hilbert function does not have ``unexpected'' plateaux.

We also record some results regarding points in $\bbP^2$. A fundamental result is the Hilbert-Burch Theorem: if $\bbX \subseteq \bbP^2$ is a set of points, then, setting $R = \Sym V^*$, the minimal free resolution of $I(\bbX)$ has the form 
\[
\textstyle 0 \to  \bigoplus_{j=1}^{\beta-1} R(-c_i) \xto{\mathrm{HB}_\bbX}  \bigoplus_{i=1}^\beta R(-b_i) \to I(\bbX) \to 0
\]
for some $\beta \geq 2$; moreover, $I(\bbX)$ is generated by the minors of size $\beta-1$ of the map $\mathrm{HB}_\bbX$, which is represented by a $\beta \times (\beta - 1)$ matrix called the \emph{Hilbert-Burch matrix} of $\bbX$.

We say that two disjoint sets of points $\bbX , \bbY$ are linked, or form a \emph{liaison}, if $\bbX \cup \bbY$ is a complete intersection. The following result is a version of \cite[Prop. 5.2.10]{MigLiaison} which relates the resolutions of $\bbX$ and the one of $\bbY$ when they form a liaison.
\begin{theorem}\label{thm: mapping cone}
Let $\bbX,\bbY \subseteq \bbP^2$ be disjoint $0$-dimensional schemes such that $\bbX \cup \bbY$ is a complete intersection of type $(d_1,d_2)$. Let 
\[
\textstyle 0 \to  \bigoplus_{j=1}^{\beta-1} R(-c_i) \to  \bigoplus_{i=1}^\beta R(-b_i) \to I(\bbX) \to 0
\]
be the minimal free resolution of $\bbX$. Then a (possibly not minimal) free resolution of $\bbY$ is given by 
\[
\textstyle 0 \to \bigoplus_{i=1}^\beta R(b_i - d_1 - d_2)  \to  R(-d_2) \oplus R(-d_1) \oplus \bigoplus_{j=1}^{\beta-1} R(c_i - d_1 - d_2) \to  I(\bbY) \to 0 .
\]
Moreover
\[
D \HF_\bbX (t) + D \HF_\bbY (d_1 + d_2 -2 - t) = D h_{\bbX \cup \bbY}(t).
\]
\end{theorem}
The maps of the resolution of $\bbY$ in~\Cref{thm: mapping cone} are defined via a snake lemma construction known as mapping cone which is not relevant for our applications. Informally, the last condition states that the sequence $D \HF_\bbY$ equals $D h_{\bbX \cup \bbY} - D \HF_\bbX$, read in reversed order.

For instance, let $\bbX$ be the set of points corresponding to the blue boxes of~\Cref{fig: example diagram}. Its minimal free resolution is 
\[
\textstyle 0 \to  R(-3) \oplus R(-5)  \to  R(-2)^{\oplus 2} \oplus R(-4) \to I(\bbX) \to 0,
\]
Let $F_2,F_4$ be generic elements of $I(\bbX)$ of degree $2$ and $4$ and let $\bbX \cup \bbY$ be the set of points defined by the complete intersection ideal $(F_2,F_4)$. In this case $Dh_{\bbX \cup \bbY} = (1,2,2,2,1)$ is represented by all boxes of~\Cref{fig: example diagram}. Applying~\Cref{thm: mapping cone}, we obtain that a resolution of $\bbY$ is given by 
\[
\textstyle 0 \to  R(-2) \oplus R(-4)^{\oplus 2}  \to  R(-1) \oplus R(-2) \oplus R(-3) \oplus R(-4) \to I(\bbY) \to 0,
\]
which is not minimal; the minimal resolution is the subcomplex $0 \to R(-4) \to R(-3) \oplus R(-1) \to I(\bbY) \to 0$. Indeed, the last statement of~\Cref{thm: mapping cone} guarantees that $D h_\bbY = (1,1,1)$, corresponding to the red boxes of~\Cref{fig: example diagram}, therefore $I(\bbY) = (G_1,G_3)$ for some polynomials of degree $1$ and $3$, respectively.

We conclude this section recalling some results on the Cayley-Bacharach property and prove a related result which will be useful in the next section. We say that a finite set of points $\bbZ \subseteq \bbP^n$ satisfies the \emph{Cayley-Bacharach property} in degree $d$ if, for every $p \in \bbZ$, we have $I(\bbZ)_d = I(\bbZ \setminus \{ p \})_d$; in other words this means that any form of degree $d$ vanishing on any $\#\bbZ-1$ points of $\bbZ$ vanishes at the last one too. It is easy to see that if $\bbZ$ satisfies the Cayley-Bacharach property in degree $d$ then it does in degree $d' \leq d$ as well, see \cite[Rmk. 4.2]{AngChiVan}.

This property is relevant in our setting because of the following result.
\begin{lemma}[{\cite[Lemma 2.25]{AC20}}] \label{lem: CB from AC}
Let $f \in S^d V$ and let $\bbX,\bbY \subseteq \bbP V$ be two disjoint irredundant decompositions of $f$. Then $\bbX \cup \bbY$ satisfies the Cayley-Bacharach property in degree $d$.
\end{lemma}

We provide the following result, which will be useful in the next section.
\begin{lemma}\label{lem: CB implies CI}
Let $a,b$ be positive integers with $a +1 < b$. Let $\bbZ \subseteq \bbP^2$ be a finite set of points such that 
    \[
    Dh_{\bbZ} = (1, 2, \ldots, a , \ldots, a, \underset{\substack{\uparrow \\ \text{deg } b}}{a-2}, a-3, \ldots, 1)
    \]
    and $\bbZ$ satisfies the Cayley-Bacharach property in degree $a+b-4$. Let $E_1$ be the unique, up to scaling, nonzero element in $I(\bbZ)_{a}$ and let $E_2$ be a generic element of $I(\bbZ)_{b}$. Then $(E_1,E_2)$ is a complete intersection. 
\end{lemma} 
\begin{proof}
Proceed by contradiction and suppose $(E_1,E_2)$ does not define a complete intersection. The structure of the Hilbert function, together with the genericity of $E_2$, implies that $I(\bbZ)_{b} = (E_1)_b + \langle E_2, E_2'\rangle$. If $(E_1,E_2)$ does not define a complete intersection, then $E_1,E_2,E_2'$ have a common factor; let $G = \gcd(E_1,E_2,E_2')$, and let $c = \deg G$. Define $\bbZ'$ to be the residue of $\bbZ$ with respect to $G$, that is $I(\bbZ') = I(\bbZ) : (G)$. Write
\[
E_1 = GF_1, \quad E_2 = GF_2, \quad E_2' = GF_2'.
\]
Then $I(\bbZ') \supseteq (F_1,F_2,F_2')$ and $(F_1,F_2)$ is a complete intersection ideal whose vanishing set is strictly larger than $\bbZ'$. This implies that $Dh_{\bbZ'}(a-c+b-c-3) = 0$, because the regularity of $\bbZ'$ must be strictly smaller than the one of a complete intersection $(a-c,b-c)$, which is $a-c+b-c -2$, see, e.g., \cite[Ch. 17]{EisCA}.

Since $\bbZ'$ imposes independent conditions in degree $a+b-2c-3$, the subspace defined by $G\cdot I(\bbZ')_{a+b-2c-3} \subseteq I(\bbZ)_{a+b-c-3}$ separates the points of $\bbZ'$; in particular $\bbZ$ does not satisfy the Cayley-Bacharach property in degree $a+b-c-3$. Since $c > 0$, we conclude that $\bbZ$ does not satisfy the Cayley-Bacharach property in degree $a+b-4$, in contradiction with the hypothesis. 
\end{proof}

\section{First cases}\label{sec: first cases}

In this section, we study the  analog of~\Cref{thm: main} and~\Cref{thm: main overcomplete} in the case of binary forms, that is when $n=1$, or ternary cubics, that is when $(n,k) = (2,1)$. These cases are in a way simpler to study, but require a more technical approach, which does not entirely follow into the framework of the Hilbert function analysis of~\Cref{sec: main proof}. We first prove a general result, in any number of variables.
\begin{proposition}\label{prop: any vars}
    Let $k,n \geq 1$. Let $f = (x_0 \cdots x_n)^k + \lambda \ell^{(n+1)k}$ for a linear form $\ell = a_0 x_0 + \cdots + a_n x_n$ with $\lambda \neq 0$ and $a_0 \cdots a_n \neq 0$. Then there exists a unique $\lambda_0$ such that $\rmR(f) = (k+1)^n - 1$ if $\lambda = \lambda_0$ and in this case the minimal decomposition of $f$ is unique. If $\lambda \neq \lambda_0$ then $\rmR(f) = (k+1)^n$.
\end{proposition}
\begin{proof}
  Let $g = (x_0 \cdots x_n)^k $ and define 
    \[
    J_{k+1} = \{ F \in \Ann(g)_{k+1} : F(\ell) = 0\} \subseteq \Ann(g)_{k+1}.
    \]
    By~\Cref{prop: monomial results}, part (i), we have $\Ann(g)_{k+1} = \langle X_0^{k+1} \vvirg X_n^{k+1} \rangle$ and $J_{k+1}$ is a hyperplane in this $n$-dimensional linear space. Since $a_0 \cdots a_n \neq 0$, the ideal $J = (J_{k+1})$ generated by $J_{k+1}$ is a complete intersection of the form of~\Cref{prop: monomial results}, part (iii). In particular, its zero set is a set of $(k+1)^n$ points, which is the unique minimal decomposition of $g$ containing $\ell$.

    Now, let $N = (k+1)^n$ and write $\bbX = \{ \ell_1 \vvirg \ell_N\}$ with $\ell = \ell_1$. We have 
    \[
   g = c_1 \ell^{(n+1)k} + c_2 \ell_2^{(n+1)k} + \cdots + c_N \ell_N^{(n+1)k},
    \]
    for some uniquely determined $c_1 \vvirg c_N$. 
    
    If $\lambda = -c_1$, then $\bbX ' = \{ \ell_2 \vvirg \ell_N\}$ defines a decomposition of $f$ of length $N-1$, showing $\rmR(f) \leq (k+1)^n - 1$. Equality holds, because $\rmR(g) = (k+1)^n$ by~\Cref{prop: monomial results}.

    If $\lambda \neq -c_1$, then $\rmR(f) = (k+1)^n$; the decomposition $\bbX$ with coefficients $(c_1 + \lambda),c_2 \vvirg c_N$ is indeed a decomposition of length $(k+1)^n$. Suppose $\bbY$ is a decomposition of length $N-1$; then $\bbY \cup \{ [\ell]\}$ would define a decomposition of $g$ of length $N$ containing $\ell$ and the argument above guarantees that $\bbX = \bbY \cup \{ [\ell]\}$, yielding a contradiction. This same argument shows that the decomposition $\bbX' = \bbX \setminus \{ [\ell]\}$ is the unique decomposition of $f$ in the case $\lambda = -c_1$.
    \end{proof}

\subsection{Binary forms}\label{sec: binary forms}
In the case of binary forms $V$ is a vector space with $\dim V = 2$. We use the notation $x,y$ for the variables spanning $V$ and $X,Y$ for their dual variables in $V^*$. Fix an integer $k$ and a linear form $\ell = ax + by$, and consider the polynomial
\[
f = x^ky^k + \ell^{2k}.
\]
Observe that if $k=1$, then $f = xy + \ell^2$ is a binary quadric with associated matrix
\[
M_f = \left[ \begin{array}{cc}
a^2 & ab + \frac{1}{2} \\
 ab + \frac{1}{2} & b^2 \\
\end{array}\right].
\]
In particular $\rmR(f) = \rank ( M_f)$, showing that $\rmR(f) = 1$ if and only if $\det M_f = 0$, or equivalently $ab = - \frac{1}{4}$. This completely characterizes the situation for binary quadrics. In particular, notice that if $ab = 0$ then $\rmR(f) = 2$.

Apolarity for binary forms yields a complete solution for the Waring problem, which dates back to \cite{Sylv} and was rediscovered in \cite{ComaSeigu}. We summarize it in the following result in the case relevant for this section. 
\begin{proposition}\label{prop: sylvester special} 
    Let $f \in S^d V$ be a binary form. Then $\Ann(f) = (F_1,F_2)$ is a complete intersection with $\deg(F_1) + \deg(F_2) = d+2$. Suppose $\deg(F_1) \leq \deg(F_2)$; then the following holds: 
    \begin{itemize}
        \item if $\Ann(f) _{\deg F_1}$ contains a square-free element, then $\rmR(f) = \deg F_1$ and any square-free element of $\Ann(f) _{\deg F_1}$ defines a decomposition of $f$;
        \item if $\Ann(f) _{\deg F_1}$ contains no square-free elements, then $\Ann(f) _{\deg F_1} = \langle F_1 \rangle$, $\rmR(f) = \deg F_2$ and any square-free elements of $\Ann(f) _{\deg F_2}$ defines a decomposition of $f$.
    \end{itemize}
\end{proposition}
Usually~\Cref{prop: sylvester special} is stated by saying that if $F_1$ is square-free then $\rmR(f) = \deg(F_1)$, otherwise $\rmR(f) = \deg(F_2)$. The statement given above is more explicit and convenient for our setting, because, in the cases of interest for us, the apolar ideal of $f$ has both generators in the same degree.

\Cref{prop: binary monomial} and~\Cref{thm: binary result} below are the analog of~\Cref{thm: main overcomplete} and~\Cref{thm: main} in the case of binary forms. 
\begin{proposition}\label{prop: binary monomial}
    The monomial $g = x^k y^k$ has irredundant decompositions of length $k+2$.
\end{proposition}
\begin{proof}
    By~\Cref{prop: monomial results}, the apolar ideal of $g$ is the complete intersection $\Ann(g) = (X^{k+1},Y^{k+1})$. A decomposition $\bbX$ of length $k+2$ is defined by a principal ideal $I(\bbX) = (F)$ and by~\Cref{lem: apolarity} we have that $F \in \Ann(g)_{k+2}$; therefore $F = L_1 X^{k+1} + L_2 Y^{k+1}$ for linear forms $L_1,L_2 \in V^*$. 
    
Since the monomials $X^{k+1} , Y^{k+1}$ have no linear syzygies, if $L_1, L_2$ are linearly independent, then $F$ is not a multiple of a polynomial of the form $X^{k+1} + \alpha Y^{k+1}$. In particular, $F$ is not divisible by an element of $\Ann(g)_{k+1}$, which guarantees that $\bbX$ is irredundant. 
\end{proof}

\begin{theorem}\label{thm: binary result}
    Let $k \geq 2$. Let $f = x^ky^k + \lambda \ell^{2k}$ for a linear form $\ell = a x + by$ and $\lambda \neq 0$. Then 
    \begin{itemize}
        \item if $ab \neq 0$, there is a unique $\lambda_0$ such that $\rmR(f) = k$ if $\lambda = \lambda_0$ and $\rmR(f) = k+1$ if $\lambda \neq \lambda_0$;
        \item if $ab = 0$ then $\rmR(f) = k+1$. 
    \end{itemize}
    \end{theorem}
\begin{proof}
The first part of the statement follows from~\Cref{prop: any vars}. 

For the second part, suppose $ab= 0$. After a suitable change of coordinates, we may assume $f = x^ky^k + x^{2k} = x^k (x^k + y^k)$. We have $\Ann(f) = (Y^{k+1}, X^{k+1} - \binom{2k}{k} Y^kX)$. By~\Cref{prop: sylvester special}, we conclude $\rmR(f) = k+1$.
\end{proof}

\subsection{Ternary cubics} 
We use the notation $x,y,z$ for the variables spanning $V$ and $X,Y,Z$ for their dual variables in $V^*$. We analyze the case $f = xyz + \ell^3$ for a linear form $\ell$. 

By~\Cref{prop: monomial results}, we have $\rmR(xyz) = 4$, which is also the generic rank in $\bbP S^3 V$. We have the following simple result.
\begin{proposition}
Let $f = xyz + \lambda \ell^3$ for a linear form $\ell = ax+by+cz$ and $\lambda \neq 0$. Then 
\begin{itemize}
    \item if $abc \neq 0$, there is a unique $\lambda_0$ such that $\rmR(f) = 3$ if $\lambda = \lambda_0$ and $\rmR(f) = 4$ if $\lambda \neq \lambda_0$;
    \item if $abc = 0$, then $\rmR(f) = 4$.
\end{itemize}
\end{proposition}
\begin{proof}
The first part of the statement follows from~\Cref{prop: any vars}. 

If $abc = 0$, then the desired claim can be verified with an explicit calculation. After possibly changing coordinates, assume $c = 0$, and rescale $x,y,z$ so that either $f = xyz + (x+y)^3$ or $f = xyz + x^3$. In both cases $\rmR(f) = 4$.
\end{proof}

\section{Ternary forms of higher degree}\label{sec: main proof}
This section gives the proof of~\Cref{thm: main overcomplete} and deduces, as a consequence,~\Cref{thm: main}.

Let $x,y,z$ be variables spanning $V$ and $X,Y,Z$ their dual variables. The proof of~\Cref{thm: main} follows from~\Cref{prop: HB liaison}, which characterizes the Hilbert-Burch matrix of a putative irredundant decomposition of $x^k y^kz^k$ of length $(k+1)^2+1$, and from the technical~\Cref{lemma: apolar identities} which in turn yields a contradiction.

\begin{proposition}\label{prop: HB liaison}
Let $g = x^k y^k z^k$ and let $\bbY$ be an irredundant decomposition of $g$ of length $(k+1)^2+1$. Then, the Hilbert-Burch matrix of $\bbY$ is
\[
\mathrm{HB}_\bbY = \left[\begin{array}{ccc}
    F_1 & F_2 & G \\
    L_1 & L_2 & Q
\end{array}\right] \qquad \text{ with } \qquad \begin{array}{ll}
    \deg(F_i) = k, & \deg(G) = k+1, \\
    \deg(L_i) = 1, & \deg(Q) = 2.
\end{array}
\]
Moreover, there is a decomposition $\bbX$ of $g$ of length $(k+1)^2$ disjoint from $\bbY$ and a scheme $\bbL$ of length $k$ contained in a line such that $\bbU = \bbX \cup \bbY \cup \bbL$ is a complete intersection $(k+1, 2k+3)$. 
\end{proposition}

\begin{proof}
    Let $\bbX$ be a decomposition of $g$ of length $(k+1)^2$ disjoint from $\bbY$. This always exists: by~\Cref{prop: monomial results}, part (iii), for any finite set of points there is a decomposition $\bbX$ of $g$ which does not intersect such a finite set. 

We are going to show that $D h_{\bbX \cup \bbY}$ has the following form:
\begin{equation}\label{eqn: XuY final}
Dh_{\bbX \cup \bbY} (t) = \left\{
\begin{array}{ll}
t +1 & \text{ if $t =  0 \vvirg k$ }\\
k+1 & \text{ if $t = k+1 \vvirg 2k+2$} \\
t'+1 & \text{ if $t = 3k+1 -t'$ with $t' = 0 \vvirg k-1$}
\end{array} \right.
\end{equation}
This is displayed in the diagram of~\Cref{fig: h_Z} in the case $k=4$, and it is represented by the union of the blue and red boxes, including the one on a different shade of red. First observe that, since $\bbX$ is a complete intersection $(k+1,k+1)$, we have $D h_{\bbX} = ( 1 \vvirg k , k \vvirg 1)$; this is represented by the blue boxes. Moreover, by assumption $\bbX$ and $\bbY$ are two decompositions of $g$, therefore $[g] \in \langle v_{3k}(\bbX) \rangle  \cap \langle v_{3k}(\bbY) \rangle$: this guarantees that $\bbX \cup \bbY$ does not impose independent conditions in degree $3k$, therefore $D h_{\bbX \cup \bbY}(3k+1) > 0$, which guarantees the presence of at least one red box in degree $3k+1$.
   
 We use~\Cref{thm: bgm} to show that $D h_{\bbX \cup \bbY}$ has no plateaux at height $h < k+1$ starting at degree $t_0 < 3k+1$. Suppose this is not the case, namely that there exists $t_0 \leq 3k$ such that $D h _{\bbX \cup \bbY}(t_0) = D h _{\bbX \cup \bbY}(t_0+1) = h < k+1$; by construction $t_0 > k$, so~\Cref{thm: bgm} applies. We deduce, in particular, that there exists a reduced curve $C$ of degree $h$ such that $D h _{(\bbX \cup \bbY)\cap C}(t) =  D h _{\bbX \cup \bbY}(t)$ for $t \geq t_0$. Let $\bar{\bbX} = \bbX \cap C$ and $\bar{\bbY} = \bbY \cap C$; note that since $\bbX$ has no equations of degree $h$, we have $\bar{\bbX} \subsetneq \bbX$. The condition $D h _{(\bbX \cup \bbY)\cap C}(t) =  D h _{\bbX \cup \bbY}(t)$ for $t \geq 3k+1 > t_0$ guarantees that 
 \[
 \langle v_{3k}(\bar{\bbX}) \rangle \cap \langle v_{3k}(\bar{\bbY}) \rangle = 
 \langle v_{3k}({\bbX}) \rangle \cap \langle  v_{3k}({\bbY}) \rangle;
 \]
therefore, we conclude that $[g] \in \langle v_{3k}(\bar{\bbX}) \rangle$, in contradiction with the minimality of the decomposition $\bbX$. 

This guarantees that there $D h_{\bbX \cup \bbY}$ has no plateaux at height $h < k+1$ starting at degree $t_0 < 3k+1$. We point out that we have not excluded the existence of a plateau starting beyond degree $3k+1$ at height $h$ with $h < Dh _{\bbX \cup \bbY}(3k+1)$. It will turn out that this does not exist either.
     
Since $Dh _{\bbX \cup \bbY}$ has no plateaux at height $h < k+1$ in the range $t = k+1 \vvirg 3k$, and $D h _{\bbX \cup \bbY} (3k+1) \geq 1$, we deduce that if $D h_{\bbX \cup \bbY}(t_0) < k+1$ then $D h_{\bbX \cup \bbY}(t)$ is strictly decreasing for $t =t_0 \vvirg 3k+1$. In particular, 
\begin{align*}
&D h _{\bbX \cup \bbY} (t) \geq k+1 \text{ for $t = k \vvirg 2k+1$ and } \\
&Dh_{\bbX \cup \bbY}(3k+1 -t') \geq t' + 1 \text{ if $t' = 0 \vvirg k$}.
\end{align*}
Pictorially, this shows that all the red and blue boxes of the diagram of~\Cref{fig: h_Z} are contained in $Dh_{\bbX \cup \bbY}$, except possibly the one in a lighter shade in position $(2k+2, k+1)$. In particular, since $\sum_{t \geq 0} Dh_{\bbX \cup \bbY}(t) = \# (\bbX\cup \bbY)$, we obtain that $Dh_{\bbX \cup \bbY}$ coincides indeed with the one of \eqref{eqn: XuY final}, except possibly for one value: informally, we have not yet determined the position of the box in a lighter shade of red.

If such last box is placed in degree $t$ for any $t = 2k+3 \vvirg 3k+2$, that is $Dh_{\bbX \cup \bbY}(2k+1+t') = k+2-t'$ for some $t' = 2 \vvirg k+1$, then $Dh_{\bbX \cup \bbY}$ would have a plateau at height $k+2-t' < k+1$, starting at degree $2k+t' \leq 3k+1$, in contradiction with the previous part of the argument.

If the last box is placed in degree $t$ for any $t = k+2 \vvirg 2k+1$, that is $Dh_{\bbX \cup \bbY}(t) = k+2$ for some $t = k+1 \vvirg 2k+1$, then $Dh_{\bbX \cup \bbY} (k+1) = k+1$, showing that $I(\bbX \cup \bbY)_{k+1} \neq 0$. Let $C$ be a curve defined by an element of $I(\bbX \cup \bbY)_{k+1}$. Since $C$ is a plane curve of degree $k+1$, for $t \geq k+1$, $Dh_{C}(t) = k+1$ and since $\bbX \cup \bbY \subseteq C$ we have $Dh_{\bbX \cup \bbY}(t) \leq k+1$ as well: this yields a contradiction. 

If the last box is placed in degree $k+1$, then $Dh_{\bbX \cup \bbY}$ has a plateau at height $k+1$ starting in degree $k+2$. Using the same argument above, by~\Cref{thm: bgm}, we would determine subsets $\bar{\bbX} \subseteq \bbX$ and $\bar{\bbY} \subseteq \bbY$ such that one of the inclusions is strict and $[g] \in \langle v_{3k} ( \bar{\bbX} ) \rangle \cap \langle v_{3k} ( \bar{\bbY} ) \rangle$. Since one of the two inclusions is strict, this is in contradiction with the irredundancy of $\bbX$ and $\bbY$.

We conclude that the last box should be placed in degree $2k+2$, that is $Dh_{\bbX \cup \bbY}(2k+2) = k+1$. In this way $Dh_{\bbX \cup \bbY}$ has the form of \eqref{eqn: XuY final} and pictorially it corresponds to the union of the blue and red boxes, including the one in a lighter shade, in~\Cref{fig: h_Z}.

\begin{figure}[htp!]
\centering
\begin{tikzpicture}[scale=0.7] 
    \draw[->, thick] (0,0) -- (16,0) node[right] {$t$};
    \draw[->, thick] (0,0) -- (0,7) node[above] {$Dh(t)$};

    \foreach \x in {0,...,15} \draw (\x + 0.5,0) -- (\x + 0.5,-0.1) node[below] {$\x$};
    \foreach \y in {1,...,6} \draw (0,\y - 0.5) -- (-0.1,\y -0.5) node[left] {$\y$};
    
    \foreach \x/\y in {0/1, 1/1, 2/1, 3/1, 4/1, 5/1, 6/1, 7/1, 8/1, 1/2, 2/2, 3/2, 4/2, 5/2, 6/2, 7/2, 2/3, 3/3, 4/3, 5/3, 6/3, 3/4, 4/4, 5/4, 4/5} 
    {
    \fill[draw=black, fill=cyan!50] (\x, \y-1) rectangle (\x+1, \y);
    }
    \foreach \x/\y in {9/1, 10/1, 11/1, 12/1, 13/1, 8/2, 9/2, 10/2, 11/2, 12/2, 7/3, 8/3,
      9/3, 10/3, 11/3, 6/4, 7/4, 8/4, 9/4, 10/4, 5/5, 6/5, 7/5, 8/5, 9/5}
    {
        \fill[draw=black, fill=red!60] (\x, \y-1) rectangle (\x+1, \y);
    }
    \fill[draw=black, fill=red!40] (10, 5 -1) rectangle (11, 5);
    \foreach \y in {1,2,3,4}
    {
        \fill[draw=black, fill=gray!50] (5 - \y + 10, \y-1) rectangle (5 - \y + 10 + 1, \y);
    }
\end{tikzpicture}
      \caption{Representation of $Dh_{\bbU}$ in the case $k=4$. The decomposition $\bbX$ is represented by the blue boxes. The union of the blue and red boxes represent $\bbX \cup \bbY$. The linked scheme $\bbL$ is represented by the gray boxes.} \label{fig: h_Z}
\end{figure}
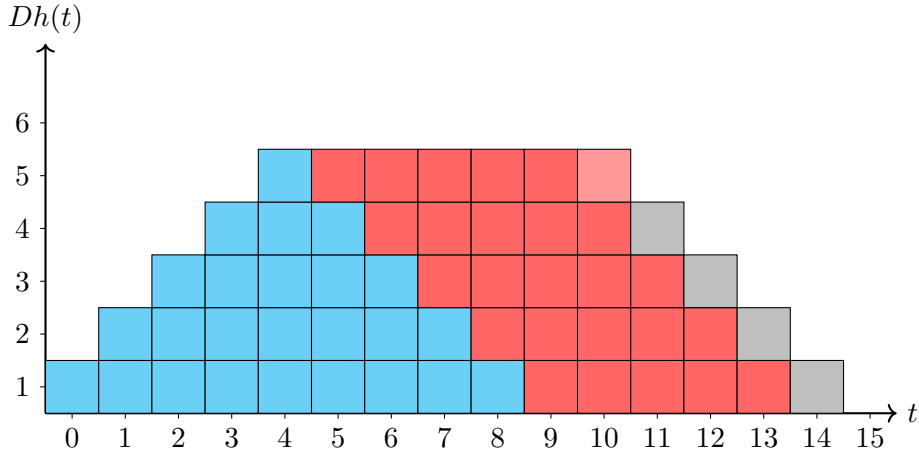

By~\Cref{lem: CB implies CI}, the characterization of $Dh_{\bbX \cup \bbY}$ implies that $I(\bbX \cup \bbY)$ contains a form $E_1$ of degree $k+1$ and a form $E_2$ of degree $2k+3$ such that $(E_1,E_2)$ is a complete intersection of type $(k+1,2k+3)$ containing $\bbX \cup \bbY$. Let $\bbU$ be the scheme defined by this complete intersection and let $\bbL$ be the residual scheme of $\bbX \cup \bbY$, that is the scheme whose ideal is $I(\bbL) = I(\bbU) : I(\bbX \cup \bbY)$. By construction $\deg(\bbL) = (k+1)(2k+3) - 2(k+1)^2 - 1 = k$. By the last statement of~\Cref{thm: mapping cone}, $Dh_{\bbL} = (1 \vvirg 1)$, therefore $\bbL$ is contained in a line. Pictorially $Dh_{\bbU}$ is represented by the union of all boxes in~\Cref{fig: h_Z} and $D h_{\bbL}$ is represented by the sequence defined by the gray boxes, in reversed order. 

We will determine the Hilbert-Burch matrix of $\bbY$ using~\Cref{thm: mapping cone} and the fact that $\bbY$ and $\bbX \cup \bbL$ are linked in $\bbU$. To apply~\Cref{thm: mapping cone}, we need to determine the resolution of $\bbX \cup \bbL$.

Let $I(\bbX) = (E_1, D_1)$ and let $L$ be the equation of the line containing $\bbL$. Then $\bbX \cup \bbL$ is contained in the complete intersection defined by $(E_1, LD_1)$ of degree $(k+1)(k+2)$; the residual scheme defined by $(E_1,LD_1) : I({\bbX \cup \bbL})$ is a single point $\{q \}$. Therefore $\bbX \cup \bbL$ and the point $\{q \}$ are linked in a complete intersection $(k+1,k+2)$. The minimal free resolution of $\{q\}$ is simply 
\[
0 \to R(-2) \to R(-1)^{\oplus 2} \to I(\{q\}) \to 0
\]
because every point is defined by the intersection of two lines. By~\Cref{thm: mapping cone}, a resolution of $\bbX \cup \bbL$ is
\[
0 \to R(-2k -2)^{\oplus 2} \to R(-2k-1) \oplus R(-k-1) \oplus R(-k-2) \to I(\bbX \cup \bbL) \to 0,
\]
and since all twists in the first syzygy module are higher than the ones in the $0$-th syzygy module, we deduce that this resolution is minimal. Now, since $\bbX \cup \bbL$ and $\bbY$ are linked in the complete intersection $\bbU$ of type $(k+1, 2k+3)$, applying again~\Cref{thm: mapping cone}, we obtain a resolution for $\bbY$ of the form 
\[
0 \to R(-k-3)\oplus R(-2k-2)\oplus R(-2k-3)\to R(-k-1) \oplus R(-k-2)^{\oplus 2} \oplus R(-2k-3) \to I(\bbY) \to 0.
\]
This resolution is not minimal. To see this, note that if the generator of degree $2k+3$ was minimal, then the syzygies involving it would be in degree higher than $2k+3$; but there are no syzygies of degree higher than $2k+3$. Therefore the minimal free resolution of $\bbY$ is 
\[
0 \to R(-k-3)\oplus R(-2k-2)\to R(-k-1) \oplus R(-k-2)^{\oplus 2} \to I(\bbY) \to 0.
\]
This yields the desired Hilbert-Burch matrix of $\bbY$.
    \end{proof}

\begin{lemma}\label{lemma: apolar identities}
 Suppose 
\[
\mathrm{HB}_\bbY = \left[\begin{array}{ccc}
    F_1 & F_2 & G \\
    L_1 & L_2 & Q
\end{array}\right] \qquad \text{ with } \qquad \begin{array}{ll}
    \deg(F_i) = k, & \deg(G) = k+1, \\
    \deg(L_i) = 1, & \deg(Q) = 2
\end{array}
\]
is a Hilbert-Burch matrix of a decomposition $\bbY$ of $g = x^k y^k z^k$ of length $(k+1)^2+1$. Then $\bbY = \bbX \cup \{ p_0 \}$ where $p_0$ is the point defined by $(L_1,L_2)$ and $\bbX$ is a decomposition of $g$ of length $(k+1)^2$.
\end{lemma}
\begin{proof}
The stabilizer of the monomial $g$ for the action of $\GL_3$ on $V = \langle x ,y,z\rangle$ is the subgroup $G_g = \bfT_{\SL} \rtimes \frakS_3$ where $\bfT_{\SL}$ is the torus of diagonal matrices with determinant $1$ and $\frakS_3$ is the permutation group embedded in $\GL_3$ as permutation matrices. The set of decompositions of a given length is therefore invariant under the action of $G_g$ and properties such as minimality and irredundancy are preserved under this action. Therefore, we may use the action of the stabilizer to normalize a decomposition.

First suppose $L_1,L_2$ are linearly dependent. After performing column operations on the Hilbert-Burch matrix, we may assume $L_2 = 0$, so that $L_1F_2 \in I(\bbY)_{k+1}$. By apolarity, $I(\bbY)_{k+1} \subseteq \langle X^{k+1}, Y^{k+1}, Z^{k+1} \rangle$, therefore, after possibly permuting the variables and rescaling $Z$, we may assume $L_1F_2 = X^{k+1} - Z^{k+1}$, arising from $L_1 = X-Z$ and $F_2 = X^k + X^{k-1}Z + \cdots + Z^k$; indeed notice that the only reducible elements in the linear span $\langle X^{k+1}, Y^{k+1}, Z^{k+1} \rangle$ are the linear combinations of two generators. Moreover, after performing column operations, we may assume that $Q$ does not depend on $Z$. However, $L_2 = 0$ also implies $QF_2 \in  I(\bbY)_{k+2}$: since $k\geq 2$, and $Q$ does not depend on $Z$, we obtain that $Q=0$. This yields a contradiction because if $Q=0$ then the ideal of maximal minors of the Hilbert-Burch matrix would not be defining a finite set of points.

Therefore, we may assume $L_1,L_2$ are linearly independent and set $p_0$ be the point which they define. After possibly acting with $G_g$, assume that $p_0$ is one of the three points $(1,1,1), (0,1,1), (0,0,1)$. After performing column operations on the Hilbert-Burch matrix, we may also assume one of these three cases occur:
\begin{enumerate}[(a)]
\item $L_1 = X-Z, L_2 = Y-Z$, and $p_0 = (1,1,1)$;
\item $L_1 = X, L_2 = Y-Z$ and $p_0 = (0,1,1)$;
\item $L_1 = X, L_2 = Y$ and $p_0 = (0,0,1)$.
\end{enumerate}
We analyze these three cases separately.

\underline{Case (a)}. After performing column operations on the Hilbert-Burch matrix, we may further assume $Q = \lambda \cdot Z^2$ for some $\lambda \in \bbC$. In this way, the Hilbert-Burch matrix has the form
\[
\mathrm{HB}_\bbY = \left[\begin{array}{ccc}
    F_1 & F_2 & G \\
    X-Z & Y-Z & \lambda Z^2
\end{array}\right]
\]
Moreover, after performing row operations, we may assume $F_1$ does not depend on $X$. Consider the minor of degree $k+1$, that is 
\[
E = (Y-Z)F_1 - (X-Z)F_2.
\]
By apolarity, we have that $E \in \langle X^{k+1}, Y^{k+1}, Z^{k+1} \rangle$. Since $ (Y-Z)F_1$ does not depend on $X$, the only term depending on $X$ in $(X-Z)F_2$ must be $X^{k+1}$. In particular, $F_2$ does not depend on $Y$ and all coefficients of the monomials $X^i Z^{k+1-i}$ for $i =1 \vvirg k$ in $(X-Z)F_2$ must vanish. This implies 
\[
F_2 = a \cdot ( X^{k} + X^{k-1}Z + \cdots + X Z^{k-1} + Z^k)
\]
for some $a \in \bbC$. Similarly 
\[
F_1 = b \cdot ( Y^{k} + Y^{k-1}Z + \cdots + Y Z^{k-1} + Z^k).
\]
Now, consider the first minor of degree $k+2$, that is 
\[
H_0 = (X-Z)G - \lambda b Z^2 (Y^{k} + Y^{k-1}Z + \cdots + Z^k);
\]
write $G = \sum_{p=0}^{k+1} G_{k+1-p}(X,Y)Z^p$ where $G_{k+1-p}$ is a homogeneous polynomial of degree $k+1-p$ depending only on $X,Y$. Then 
\begin{align*}
H_0 =   &Z^0[XG_{k+1}(X,Y)] + \\
        &Z^1[XG_k(X,Y)-G_{k+1}(X,Y)] + \\
        &Z^2[XG_{k-1}(X,Y)-G_k(X,Y)  - b\lambda \cdot Y^k] + \mathit{higher \ order \ terms \ in \ Z}.
\end{align*}
By apolarity, $H_0 \in (X^{k+1}, Y^{k+1}, Z^{k+1})$. Since $k \geq 2$ and $(X^{k+1}, Y^{k+1}, Z^{k+1})$ is a monomial ideal, we obtain the three terms in the square brackets belong to $(X^{k+1}, Y^{k+1}, Z^{k+1})$ as well.

In particular, $XG_{k+1}(X,Y) \in (X^{k+1}, Y^{k+1})$ implies that 
\[
G_{k+1} = u_0 X^{k+1} + u_1 X^{k}Y + u_{k+1} Y^{k+1}
\]
for some $u_i \in \bbC$. The condition $XG_k(X,Y)-G_{k+1}(X,Y) \in (X^{k+1}, Y^{k+1})$, together with the form of $G_{k+1}$ implies 
\[
G_k = v_0 X^k + u_1 X^{k-1}Y.
\]
The condition $XG_{k-1}(X,Y)-G_k(X,Y) + a\lambda \cdot Y^k \in (X^{k+1}, Y^{k+1})$ implies 
\[
XG_{k-1}(X,Y)-G_k(X,Y)  - b\lambda  \cdot Y^k = 0
\]
because $(X^{k+1}, Y^{k+1})$ is zero in degree lower than $k+1$; because of the forms we deduced for $G_k$, and since $XG_{k-1}$ is a multiple of $X$, we obtain $b\lambda = 0$.

If $\lambda \neq 0$, then $b = 0$ which implies $F_1 = 0$. The same calculation on the equations $H_1$ implies $a = 0$ and $F_2 = 0$. In this case, the maximal minors of the Hilbert-Burch matrix do not generate a $0$-dimensional ideal, yielding a contradiction.

If $\lambda = 0$, we deduce that the Hilbert-Burch matrix of $\bbY$ has the form 
\[
\mathrm{HB}_\bbY = \left[\begin{array}{ccc}
    F_1 & F_2 & G \\
    X-Z & Y-Z & 0
\end{array}\right];
\]
its ideal of maximal minors is $I(\bbY) = (E,(X-Z)G,(Y-Z)G)$ where $E = (Y-Z)F_1 - (X-Z)F_2$. In particular $p_0 \in \bbY$.

Considering the minors of the Hilbert-Burch matrix, apolarity guarantees 
\begin{align*}
(X-Z)G , \ (Y-Z)G \ \in \ (X^{k+1}, Y^{k+1}, Z^{k+1});
\end{align*}
an argument similar to the one above implies $G \in (X^{k+1}, Y^{k+1}, Z^{k+1})$. In this case, we obtain 
\[
I(\bbY) = (E , (X-Z)G ,(Y-Z)G) = (X-Z, Y-Z) \cap (E,G).
\]
Let $\bbX$ be the zero set defined by the ideal $(E,G)$; since $E,G \in (X^{k+1}, Y^{k+1}, Z^{k+1})$, we have that $\bbX$ is a decomposition of $g$ of length $(k+1)^2$. We conclude $\bbY = \bbX \cup \{p_0\}$ which proves the desired condition.

\underline{Case (b)}.
As before, after performing column operations on the Hilbert-Burch matrix, we may further assume $Q = \lambda \cdot Z^2$ for some $\lambda \in \bbC$. In this way, the Hilbert-Burch matrix has the form
\[
\mathrm{HB}_\bbY = \left[\begin{array}{ccc}
    F_1 & F_2 & G \\
    X & Y-Z & \lambda Z^2
\end{array}\right].
\]
Moreover, after performing row operations, we may assume $F_1$ does not depend on $X$. Consider the minor of degree $k+1$, that is 
\[
E = (Y-Z)F_1 - XF_2.
\]
By apolarity, we have that $E \in (X^{k+1}, Y^{k+1}, Z^{k+1})$. Since $ (Y-Z)F_1$ does not depend on $X$, the only term depending on $X$ in $XF_2$ must be $X^{k+1}$. This guarantees $F_2 = a X^k$ and the same argument as in the previous part guarantees 
\[
F_1 = b \cdot ( Y^{k} + Y^{k-1}Z + \cdots + Y Z^{k-1} + Z^k).
\]
Now, consider the first minor of degree $k+2$, that is 
\[
H_0 = XG - \lambda b Z^2 (Y^{k} + Y^{k-1}Z + \cdots + Z^k);
\]
write $G = \sum_{p=0}^{k+1} G_{k+1-p}(X,Y)Z^p$ where $G_{k+1-p}$ is a homogeneous polynomial of degree $k+1-p$ depending only on $X,Y$. Then 
\begin{align*}
H_0 =   &Z^0[XG_{k+1}(X,Y)] + \\
        &Z^1[XG_k(X,Y)] + \\
        &Z^2[XG_{k-1}(X,Y)- b\lambda \cdot Y^k] + \mathit{higher \ order \ terms}.
\end{align*}
Similarly to the previous case the apolarity condition implies $XG_{k-1}(X,Y)- b\lambda \cdot Y^k = 0$ which implies $b\lambda = 0$. The condition on $H_1$, on the other hand, implies $a\lambda = 0$. As before, if $\lambda \neq 0$, we have $a=b=0$ and we obtain a contradiction because the maximal minors of the Hilbert-Burch matrix would not generate a $0$-dimensional ideal. 

If $\lambda = 0$, we deduce that the Hilbert-Burch matrix of $\bbY$ has the form 
\[
\mathrm{HB}_\bbY = \left[\begin{array}{ccc}
    F_1 & F_2 & G \\
    X  & Y-Z & 0
\end{array}\right];
\]
its ideal of maximal minors is $I(\bbY) = (E,XG,(Y-Z)G)$ where $E = (Y-Z)F_1 - XF_2$. In particular $p_0 \in \bbY$.

Considering the minors of the Hilbert-Burch matrix, apolarity guarantees 
\begin{align*}
&XG, \ (Y-Z)G \ \in  \ (X^{k+1}, Y^{k+1}, Z^{k+1}) ;
\end{align*}
an argument similar to the one above implies $G \in (X^{k+1}, Y^{k+1}, Z^{k+1})$. In this case, we obtain 
\[
I(\bbY) = (E , XG ,(Y-Z)G) = (X, Y-Z) \cap (E,G).
\]
Let $\bbX$ be the zero set defined by the ideal $(E,G)$; since $E,G \in (X^{k+1}, Y^{k+1}, Z^{k+1})$, we have that $\bbX$ is a decomposition of $g$ of length $(k+1)^2$. We conclude $\bbY = \bbX \cup \{p_0\}$ which proves the desired condition.

\underline{Case (c)}. The calculation is similar to the previous cases and the proof is completed.
\end{proof}

\Cref{prop: HB liaison} and~\Cref{lemma: apolar identities} provide the contradiction needed to complete the proof of~\Cref{thm: main overcomplete}.
\begin{proof}[{Proof of~\Cref{thm: main overcomplete}}]
Proceed by contradiction and assume that $g = x^ky^kz^k$ has an irredundant decomposition of length $(k+1)^2+1$. Then~\Cref{prop: HB liaison} characterizes the Hilbert-Burch matrix of such decomposition, and~\Cref{lemma: apolar identities} shows that a decomposition with such Hilbert-Burch matrix is not irredundant. 

This provides a contradiction and completes the proof.
\end{proof}

And finally, we can complete the proof of~\Cref{thm: main}, which is a consequence of~\Cref{thm: main overcomplete} and~\Cref{prop: any vars}.
\begin{proof}[{Proof of~\Cref{thm: main}}]
 If $abc \neq 0$, then~\Cref{prop: any vars} guarantees there is a unique $\lambda$ such that $\rmR(f) = (k+1)^2 - 1$ and in that case the decomposition is unique.

Now suppose $\lambda$ does not coincide with the coefficient $c_1$ of the decomposition of $g = x^ky^kz^k$ in the rank decomposition $\bbX$ containing $\lambda$. Suppose $\bbY$ is a decomposition of $f$ different from $\bbX$. Then $\bbY' = \bbY \cup \{ [\ell]\}$ would define a decomposition of $g$ of length $(k+1)^2 + 1$, where the coefficient of $\ell^{3k}$ is $\lambda \neq 0$. By~\Cref{thm: main overcomplete}, such decomposition cannot be irredundant, but by construction all powers appear with nonzero coefficients. This yields a contradiction.

Now assume $abc = 0$, and without loss of generality assume that either $c = 0$ or $b=c=0$. To prove the lower bound, we use the same argument as above. If $\rmR(f) = (k+1)^2$ then a decomposition $\bbX$ of length $(k+1)^2$, together with the point $[ \ell ]$, defines an irredundant decomposition of the monomial $g$; by~\Cref{prop: monomial results}, there is no minimal decomposition of $g$ containing $\ell$, so $\bbX \cup \{[\ell]\}$ has cardinality $(k+1)^2+1$. This provides an irredundant decomposition of $g$ of length $(k+1)^2 + 1$, in contradiction with~\Cref{thm: main}. 
\end{proof}

\subsection*{Acknowledgments} L.C. is a member of GNSAGA of INdAM. Part of this work was developed while S.M. was a visiting scholar at the Institut de Mathematiques de Toulouse. 

{
\bibliographystyle{alphaurl}
\bibliography{binomials}
}

\end{document}